\documentclass[11pt]{amsart}
\usepackage{amsmath}
\usepackage{amsxtra}
\usepackage{amscd}
\usepackage{amsthm}
\usepackage{amsfonts}
\usepackage{amssymb}
\usepackage{eucal} 
\usepackage{url}
\usepackage{setspace}  
\input xy
\xyoption{all}

\newtheorem{thm}{Theorem}
\newtheorem{cor}[thm]{Corollary}
\newtheorem{lem}[thm]{Lemma}
\newtheorem{prop}[thm]{Proposition}

\theoremstyle{remark}
\newtheorem{rem}[thm]{Remark}
\newtheorem{exam}[thm]{Example}

\theoremstyle{definition}
\newtheorem{defn}[thm]{Definition}

\newcommand\nc{\newcommand}
\nc\on{\operatorname}

\newcommand\fp{{\mathfrak p}}
\newcommand\fq{{\mathfrak q}}
\newcommand\fm{{\mathfrak m}}

\newcommand\mcl{\mathcal}
\newcommand\mbb{\mathbb}

\nc\Hom{\on{Hom}}
\nc\Sym{\on{Sym}}
\nc\Spec{\on{Spec}}
\nc\Specm{\on{Specm}}
\nc{\dfp}{\overset{\cdot}{\fp}}
\nc{\dfq}{\overset{\cdot}{\fq}}
\nc{\dfm}{\overset{\cdot}{\fm}}

\begin{document}

\title{Zeta Functions of Curves With No Rational Points}
\author{Daniel Litt}

\begin{abstract}
We show that the motivic zeta functions of smooth, geometrically connected curves with no rational points are rational functions.  This was previously known only for curves whose smooth projective models have a rational point on each connected component.  In the course of the proof we study the class of a Severi-Brauer scheme over a general base in the Grothendieck ring of varieties.  
\end{abstract}

\maketitle

\section{Introduction}\label{Introduction}
Let $k$ be a field, and $K_0(\on{Var}_k)$ the Grothendieck ring of varieties over $k$.  This is the free abelian group on isomorphism classes $[X]$ of separated, finite type $k$-schemes (varieties), subject to the following relation:  $$[X]=[Y]+[X\setminus Y] \text{ for $Y\hookrightarrow X$ a closed embedding}.$$
Multiplication is given by $$[X]\cdot [Y]=[X\times Y]$$ on classes of varieties, and extended linearly.  This ring was introduced by Grothendieck \cite[(Letter of August 16, 1964)]{Groth} in a letter to Serre.  The Grothendieck ring of varieties is the universal ring through which all ``motivic" invariants factor (e.g.~Euler characteristic with compact support, Hodge-Deligne polynomial, virtual Chow motive, etc.).  

Let $\mathbb{L}:=[\mathbb{A}^1]\in K_0(\on{Var}_k)$ be the class of the affine line.  
\begin{exam}
Using the fact that $\mbb{P}^n=\on{pt}\cup~\mbb{A}^1 \cup \mbb{A}^2\cup \cdots \cup \mbb{A}^n$, we have $$[\mbb{P}^n]=1+\mbb{L}+\cdots +\mbb{L}^n.$$
\end{exam}

In \cite[1.3]{kapranov}, Kapranov introduces for each variety $X/k$ a motivic zeta function $Z_X(t)$.
\begin{defn}[Kapranov motivic zeta function]\label{kapranovzeta}
Let $X$ be a $k$-variety.  Then the motivic zeta function $Z_X(t)\in K_0(\on{Var}_k)[[t]]$ is
$$Z_X(t):=\sum_{n=0}^\infty [\on{Sym}^n(X)]t^n.$$  
\end{defn}
\begin{rem}\label{mult}
The motivic zeta function is a homomorphism $$K_0(\on{Var}_k)\to 1+tK_0(\on{Var}_k)[[t]];$$ that is, if $[X]=[Y]+[W]$, then $Z_X(t)=Z_Y(t)\cdot Z_W(t)$.
\end{rem}
\begin{exam}[{\cite[Corollary 3.6]{rationalitycriteria}}]\label{Pn}
$$Z_{\mbb{P}^n}(t)=\frac{1}{(1-t)(1-\mbb{L}t)\cdots (1-\mbb{L}^nt)}.$$
\end{exam}
If $k=\mathbb{F}_q$ is a finite field, $Z_X(t)$ is an analogue of the Weil zeta function $$\zeta_X(t):=\exp\left(\sum_{k=1}^\infty \frac{\#|X(\mathbb{F}_{q^k})|}{k}t^k\right)= \sum_{n=0}^\infty \#|\on{Sym}^n(X)(\mathbb{F}_q)|t^n. $$  Indeed, in this case there is a natural homomorphism $$\#: K_0(\on{Var}_k)\to \mbb{Z}$$  $$[X]\mapsto \#X(\mbb{F}_q);$$ and $\#(Z_X(t))=\zeta_X(t).$ Kapranov shows

\begin{prop}[{\cite[1.1.9]{kapranov}}, {\cite[Theorem 7.33]{mustata}}] \label{rationalitywithpoints} Let $k$ be a field and $C/k$ a smooth, geometrically connected, projective curve of genus $g$ with $C(k)\not=\emptyset$.  Then $$Z_C(t)(1-t)(1-\mathbb{L}t)\in K_0(\on{Var}_k)[[t]]$$
 is a polynomial of degree $2g$.
\end{prop}
This result is analogous to (and implies) the rationality of the Weil zeta function of $C$, if $k$ is a finite field, by applying $\#(-)$.
\begin{rem}
Kapranov speculates that $Z_X(t)$ may be a rational function for arbitrary $k$-varieties $X$ \cite[Remark 1.3.5(b)]{kapranov}.  If $k$ is finite, such a result would give a geometric explanation for the rationality of the Weil zeta function $\zeta_X$.  However, Larsen and Lunts show that for $k=\mbb{C}$ and $X$ is a surface with Kodaira dimension different from $-\infty$, $Z_X(t)$ is not rational \cite[Theorem 1.1]{rationalitycriteria}.  The problem of finding a natural quotient of $K_0(\on{Var}_k)$ (through which ``motivic" invariants still factor) over which $Z_X(t)$ becomes rational is of some importance.
\end{rem}
Kapranov remarks that $Z_C(t)$ is still a rational function if $C(k)=\emptyset$; however a correct proof of this fact has not yet appeared in the literature.  The main reason for writing the present note was to rectify this lack, as the proof is not a triviality. 

Let us briefly review the proof of Proposition \ref{rationalitywithpoints}, and then discuss how it fails if $C(k)=\emptyset.$  
\begin{proof}[Proof of Proposition \ref{rationalitywithpoints}]
Observe that if $C(k)\not=\emptyset$, the Abel-Jacobi map $\on{Sym}^n(C)\to \on{Pic}^n(C)$ is a (Zariski) $\mathbb{P}^{n-g}$-bundle for $n>2g-2$ \cite[Theorem 4]{schwarzenberger}; thus for $n>2g-2$, $$[\on{Sym}^n(C)]=[\mathbb{P}^{n-g}][\on{Pic}^n(C)]=\frac{1-\mathbb{L}^{n-g+1}}{1-\mathbb{L}}[\on{Pic}^n(C)].$$  Furthermore, $\on{Pic}^n(C)\simeq \on{Pic}^0(C)$ for all $n$ (again using the existence of a rational point on $C$).  In particular,
$$Z_C(t) = \sum_{n=0}^{2g-2} [\on{Sym}^n(C)]t^n+[\on{Pic}^0(C)]\sum_{n=2g-1}^\infty \frac{1-\mathbb{L}^{n-g+1}}{1-\mathbb{L}}t^n$$
and thus $$(1-t)(1-\mathbb{L}t)Z_C(t)$$ is a polynomial of degree $2g$.
\end{proof}
\begin{exam}\label{realcurve}
Unfortunately, the first step of this proof breaks if $C(k)=\emptyset$.  For example, consider the curve $X$ in $\mathbb{P}^2_{\mathbb{R}}$ defined in homogeneous coordinates by $$x^2+y^2+z^2=0.$$  It is easy to see that $\on{Pic}^n(X)=\on{Spec}(\mathbb{R})$ for all $n$ (see e.g.~\cite[9.2.4]{FGA explained}), but $\on{Sym}^n(X)$ has no rational points if $n$ is odd.  Thus the Abel-Jacobi morphism $\on{Sym}^n(X)\to \on{Pic}^n(X)$ is not a Zariski $\mathbb{P}^n$-bundle for odd $n$.  
\end{exam}
\begin{rem}
Theorem \ref{maintheorem} of this note implies that in the above example, $(1-\mbb{L}^2t^2)(1-t^2)Z_X(t)$ is a polynomial.  Let us compare this to Remark 1.3.5(a) of \cite{kapranov}.  The remark states that $(1-\mbb{L}^nt^n)(1-t^n)Z_X(t)$ is a polynomial, where $n>0$ is minimal such that $\on{Pic}^n(X)(k)\not=\emptyset$; in the above example $\on{Pic}^1(X)=\on{Spec}(\mbb{R})$, so the remark suggests that $(1-\mbb{L}t)(1-t)Z_X(t)$ is rational.  We do not know a proof of this fact, and do not believe it to be true (though we do not have a proof it is false).  

There is some ambiguity in Remark 1.3.5(a) of \cite{kapranov}, which may allow one to preserve its correctness. In the case $X$ has no rational points, the scheme $\on{Pic}(X)$ represents the fppf sheafification of the functor sending $T$ to the set of isomorphism classes of line bundles on $X\times T$, modulo line bundles pulled back from $T$.  If we take the comment to refer to the Zariski sheafification of this functor instead, the remark has some chance of being true---though again, we do not know a proof.
\end{rem}

The issue identified in Example \ref{realcurve} is that $\on{Sym}^n(C)\to \on{Pic}^n(C)$ may not be a Zariski fiber-bundle.  Of course (if $C$ is geometrically connected), after a finite extension of the base field, we recover the usual situation of a projective space bundle over the $\on{Pic}^n(C)$, so in general $\on{Sym}^n(C)\to\on{Pic}^n(C)$ is a Severi-Brauer scheme over $\on{Pic}^n(C)$.  Thus we will proceed by studying the class $[V]$ of a Severi-Brauer $S$-scheme $V/S$ in $K_0(\on{Var}_k)$.  The main result of this study is a description of the class $[V]\in K_0(\on{Var}_k)$.
\begin{thm}\label{sbresult}
Let $S$ be a $k$-variety and $\alpha\in \on{Br}(S)$ a Brauer class.  Then there is an element $P=P(\alpha, S)\in K_0(\on{Var}_k)$ and an integer $r=r(\alpha, S)$ determined only by $\alpha$ and $S$, such that for any Severi-Brauer $S$-scheme $V$ with Brauer class $\alpha$, $$[V]=P(1+\mbb{L}^r+\mbb{L}^{2r}+\cdots+\mbb{L}^{nr})$$ for some $n$.  
\end{thm}
See Proposition \ref{refined9} for a refined version of this result.

After giving this description of $[V]$, we will prove the main result of the paper:
\begin{thm}\label{maintheorem}
Let $C$ be a smooth, projective, geometrically connected curve over a field $k$.  Then there exists a polynomial $$p(t)\in 1+tK_0(\on{Var}_k)[t]$$ such that $p(t)Z_C(t)\in K_0(\on{Var}_k)[[t]]$ is a polynomial with constant term $1$.
\end{thm}
\begin{rem}
Informally, we say that $C$ has ``rational motivic zeta function."  That $p(t)$ and $p(t)Z_C(t)$ have constant term $1$ is important: it implies that the numerator and denominator of this rational function are invertible in $K_0(\on{Var}_k)[[t]]$.  
\end{rem}
\begin{rem}
Much previous work (\cite{zhykhovich}, \cite{Kar95}) studies the Chow motive of a Severi-Brauer variety.  This methods of this note may be used to recover many of the results of these works; we believe our methods bear some similarity to those of \cite{Kar95}.
\end{rem}
\subsection{Acknowledgements}
This note would not have been possible without discussions with Jeremy Booher, A.J.~de Jong, Mircea Musta\c{t}\u{a}, John Pardon, Burt Totaro, and my advisor Ravi Vakil.
\section{Twisted Sheaves}
Traditionally, the Brauer group of a scheme is studied by means of Azumaya algebras (\cite{groupedebraueri}, \cite{groupedebrauerii}) or Severi-Brauer varieties \cite{artin}; we will find it convenient to use the notion of twisted sheaves instead.

We begin with a brief, largely self-contained, review of the facts about twisted sheaves that we will need---useful references are Caldararu \cite{caldararu} or Lieblich \cite{lieblich}.
 
\begin{defn}[The category of $\alpha$-twisted sheaves, $\on{QCoh}(X, \alpha)$]
Let $X$ be a scheme and $\alpha\in H^2(X_{\text{\'et}}, \mathbb{G}_m)$ a cohomology class, represented by a \v{C}ech 2-cocycle $\lambda\in \Gamma(U\times_X U\times_X U, \mbb{G}_m)$ for some \'etale cover $U\to X$.  The objects of the category $\on{QCoh}(X, \alpha)$ of $\alpha$-twisted sheaves are ``descent data for quasi-coherent sheaves," twisted by $\alpha$.  Namely, let $\pi_1, \pi_2: U\times_X U\to U$ be the two projections, and similarly with $\pi_{ij}: U\times_X U\times_X U\to U\times_X U$.  An $\alpha$-twisted sheaf is the data of a quasi-coherent sheaf $\mathcal{E}$ on $U$, and an isomorphism $\phi: \pi_1^*\mathcal{E}\overset{\sim}{\to} \pi_2^*\mathcal{E}$, so that $\pi_{23}^*\phi\circ \pi_{12}^*\phi=\lambda\cdot \pi_{13}^*\phi$.  Observe that if $\mathcal{E}$ is a vector bundle, we may (after refining $U$ to trivialize $\mathcal{E}$) view this descent data as the data of a section $g'\in \Gamma(U\times_X U, \on{GL}_n)$; we call $(\mathcal{E}, \phi)$ an $\alpha$-twisted vector bundle if $\mcl{E}$ is a vector bundle.

A morphism $(\mathcal{E}, \phi)\to (\mathcal{E}', \phi')$ is defined to be a morphism $f: \mcl{E}\to \mcl{E}'$ so that the diagram $$\xymatrix{
\pi_1^*\mathcal{E} \ar[r]^\phi \ar[d]^{\pi_1^*f} & \pi_2^*\mathcal{E}\ar[d]^{\pi_2^*f}\\
\pi_1^*\mathcal{E}' \ar[r]^{\phi'} & \pi_2^*\mathcal{E}'
}$$
commutes.
\end{defn}
\begin{rem}
A priori, the definition of $\on{QCoh}(X, \alpha)$ depends on the choice of cocycle $\lambda$ representing $\alpha\in H^2(X, \mbb{G}_m)$.  However, if $\lambda$ and $\lambda'$ are two cocycles representing $\alpha$, then the categories of twisted sheaves they define are (non-canonically) equivalent.  Namely, choose a $1$-cocycle $\beta$ with $d\beta=\lambda^{-1}\lambda'$.  Then the functor $$(\mcl{E}, \phi)\mapsto (\mcl{E}, \beta\phi)$$ is an equivalence of categories $$\on{QCoh}(X, [\lambda])\to \on{QCoh}(X, [\lambda']).$$   
This equivalence \emph{does} depend on the choice of $\beta$; these equivalences (up to natural isomorphism) are a torsor under $H^1(X, \mbb{G}_m)$ (which corresponds to the fact that there are autoequivalences of $\on{QCoh}(X, \alpha)$ coming from the functors $$(\mathcal{E}, \phi)\mapsto (\mathcal{E}\otimes \mcl{L}, \phi\otimes \on{id})$$ where $\mcl{L}$ is a line bundle on $X$).
\end{rem}

\begin{prop}\label{twistedfacts}
Let $\alpha, \alpha'\in H^2(X_{\text{\'et}}, \mathbb{G}_m)$ be cohomology classes.
\begin{enumerate}
\item $\alpha$ is a Brauer class if and only if there exists an $\alpha$-twisted vector bundle.  
\item $\on{QCoh}(X, \alpha)$ is an Abelian category with enough injectives.
\item There are natural functors $$-\otimes-: \on{QCoh}(X, \alpha)\times \on{QCoh}(X, \alpha')\to \on{QCoh}(X, \alpha+\alpha')$$ and $$\on{Hom}(-,-): \on{QCoh}(X, \alpha)^{op}\times \on{QCoh}(X, \alpha')\to \on{QCoh}(X, \alpha'-\alpha)$$ given by $\otimes$ and $\on{Hom}$ on twisted descent data.
\item Similarly, $\bigwedge^n$ and $\on{Sym}^n$ extend to functors $\on{QCoh}(X, \alpha)\to \on{QCoh}(X, n\alpha)$.
\item If $f: X\to Y$ is a morphism, there is a natural functor $$f^*: \on{QCoh}(Y, \alpha)\to \on{QCoh}(X, f^*\alpha)$$ given by applying $f^*$ to twisted descent data.
\item $\on{QCoh}(X, 0)$ is the usual category of quasi-coherent sheaves on $X$.
\end{enumerate}
\end{prop}
\begin{proof}
All of the statements are easy, aside from (2).  For a sketch proof of (2), see \cite[Lemma 2.2.3.2]{lieblich} or \cite[Lemma 2.1.1]{caldararu}.
\end{proof}
\begin{prop}\label{twistedlinebundle}
There is an $\alpha$-twisted line bundle on $X$ if and only if $\alpha=0$.  
\end{prop}
\begin{proof}
If $\alpha=0$, $\mathcal{O}_X$ is an $\alpha$-twisted line bundle.

On the other hand, let $(\mathcal{L}, \phi)$ be an $\alpha$-twisted line bundle, given by twisted descent data on some \'etale cover $U\to X$.  We may choose a cover $r: U'\to U$ so that $r^*\mathcal{L}$ is trivial; after choosing a trivialization, we may view $r^*\phi$ as an element of $\Gamma(U'\times_X U', \mathbb{G}_m)$, e.g. a $1$-cochain for $\mathbb{G}_m$.  But then $[d(r^*\phi)]=\alpha$, so $\alpha$ is a coboundary.  Thus $\alpha=0$.
\end{proof}
\begin{cor}
Suppose $\mathcal{E}$ is an $\alpha$-twisted vector bundle of rank $n$.  Then  $\alpha$ is $n$-torsion in $Br(X)$.
\end{cor}
\begin{proof}
$\bigwedge^n\mathcal{E}$ is an $n\alpha$-twisted line bundle by Proposition \ref{twistedfacts}(4)---thus $n\alpha=0$ in $Br(X)$ by Proposition \ref{twistedlinebundle}.
\end{proof}

Twisted vector bundles have many of the same properties of vector bundles.
\begin{prop}\label{semisimple}
Suppose $X$ is an affine scheme, and $\alpha$ a Brauer class on $X$.  Then all short exact sequences of $\alpha$-twisted vector bundles on $X$ split.
\end{prop}
\begin{proof}
Suppose $$0\to \mathcal{E}_1\to \mathcal{E}_2\to \mathcal{E}_3\to 0$$ is a short exact sequence of $\alpha$-twisted vector bundles.  We wish to show that $\on{Ext}^1_{\on{QCoh}(X, \alpha)}(\mathcal{E}_3, \mathcal{E}_1)=0$.  But we have $$\on{Ext}^1_{\on{QCoh}(X, \alpha)}(\mathcal{E}_3, \mathcal{E}_1)=H^1(X_{\text{\'et}}, \mathcal{E}_3^\vee\otimes \mathcal{E}_1)=0$$ where we use that $X$ is affine and that \'etale cohomology of coherent sheaves is the same as Zariski cohomology.
\end{proof}
\begin{cor}\label{schurlemma}
Let $\mathcal{E}$ be an $\alpha$-twisted vector bundle over the spectrum of a field.  Then $\mathcal{E}$ is simple if and only if $\on{End}(\mathcal{E})$ is a division algebra.
\end{cor}
\begin{proof}
Suppose $\mathcal{E}$ is simple.  Then any non-zero endomorphism of $\mathcal{E}$ must have no kernel, as the kernel would be an $\alpha$-twisted sub-bundle of $\mathcal{E}$.  But we are working over a field, so (working \'etale-locally), we see that an endomorphism with no kernel is an isomorphism.

On the other hand, if $\mathcal{E}$ is not simple, Proposition \ref{semisimple} gives that $\mathcal{E}=\mathcal{F}\oplus \mathcal{G}$ for some non-zero $\mcl{F}, \mcl{G}$; then projection to either factor is a non-invertible endomorphism.
\end{proof}

\begin{cor}\label{div}
Let $X$ be the spectrum of a field.  Then there is a unique isomorphism class of non-zero simple $\alpha$-twisted vector bundle over $X$.
\end{cor}
\begin{proof}
Suppose $D, D'$ are non-zero simple $\alpha$-twisted vector bundles.  Then $$\on{Hom}_{\on{QCoh}(X, \alpha)}(D, D')\simeq H^0_{\on{QCoh}(X)}(D^\vee \otimes D')\not=0.$$  But as $D, D'$ are simple, any non-zero morphism between them is an isomorphism.
\end{proof}
\begin{cor}[Artin-Wedderburn]\label{artinwedderburn}
Let $X$ be the spectrum of a field, and $D$ the unique non-zero simple $\alpha$-twisted vector bundle over $X$.  Then any $\alpha$-twisted vector bundle $E$ is isomorphic to $D^{\oplus n}$ for some $n$.
\end{cor}
\begin{proof}
Let $E$ be a non-zero $\alpha$-twisted vector bundle.  If $E$ is simple, it is isomorphic to $D$ by Corollary \ref{div}.  Otherwise, let $E'$ be a nonzero proper sub-bundle; by induction on the rank, $E'\simeq D^{\oplus k}$ and $E/E'\simeq D^{\oplus k'}$.  So there is a short exact sequence $$0\to D^{\oplus k}\to \mcl{E}\to D^{\oplus k'}\to 0$$ and we may conclude the corollary by Proposition \ref{semisimple}.
\end{proof}
\begin{cor}\label{monomorphism}
Let $X$ be an integral Noetherian scheme and $\mathcal{E}_1, \mathcal{E}_2$ be two $\alpha$-twisted vector bundles on $X$, with ranks $r_1\leq r_2$.  Then there exists a non-empty open set $U\subset X$ and a monomorphism $\iota: \mathcal{E}_1|_U\hookrightarrow \mathcal{E}_2|_U$ so that $\on{coker}(\iota)$ is an $\alpha$-twisted vector bundle.
\end{cor}
\begin{proof}
We apply Corollary \ref{artinwedderburn} at the generic point $\eta$ of $X$ to obtain a monomorphism $\mathcal{E}_1|_\eta\hookrightarrow \mathcal{E}_2|_\eta$.  Spreading out gives the claim.
\end{proof}

If $\mathcal{E}$ is a vector bundle, one may consider $\mbb{P}\mathcal{E}$, the scheme of hyperplanes in $\mathcal{E}$.  Similarly, given an $\alpha$-twisted sheaf $\mathcal{E}$ over a scheme $X$, one may obtain a Severi-Brauer variety with Brauer class $\alpha$ by considering $\mathbb{P}\mathcal{E}$, which is \'etale descent data for a scheme over $X$.  As $\mathbb{P}\mathcal{E}$ is anti-canonically polarized over $X$, this descent data is effective and we obtain a Severi-Brauer variety over $X$.  To obtain an Azumaya algebra with Brauer class $\alpha$, simply consider $\on{End}(\mathcal{E})$.  It is not hard to see that every Severi-Brauer variety or Azumaya algebra is obtained in this fashion; indeed, take the $\on{PGL}_n$-cocycle defining the Severi-Brauer variety or Azumaya algebra, and lift it to an arbitrary cocycle for $\on{GL}_n$. (To do so, one may have to refine the cover on which the cocycle is defined.)

We will require the following well-known fact about Severi-Brauer schemes; we sketch a proof using twisted sheaves.
\begin{cor}\label{sbwithsection}
Let $\pi: \mbb{P}\to S$ be a Severi-Brauer scheme over $S$.  If $\pi$ admits a section, $\mbb{P}=\mbb{P}(\mcl{E})$ for $\mcl{E}$ a vector bundle over $S$.
\end{cor}
\begin{proof}
Let $\mcl{E}$ be an $\alpha$-twisted vector bundle so that $\mbb{P}=\mbb{P}(\mcl{E})$; we wish to show that $\alpha=0\in H^2(S, \mbb{G}_m)$.  But the section to $\pi$ corresponds to an $\alpha$-twisted line sub-bundle of $\mcl{E}$; hence by Proposition \ref{twistedlinebundle}, $\alpha$ is trivial.
\end{proof}

\section{The Motive of a Severi-Brauer Variety}
Suppos $S$ is a $k$-variety and $$0\to \mathcal{E}_1\to \mathcal{E}_2\to \mathcal{E}_3\to 0$$ is a short exact sequence of $\alpha$-twisted vector bundles on $S$.  We wish to relate the classes of the Severi-Brauer schemes $$\mathbb{P}(\mathcal{E}_1), \mathbb{P}(\mathcal{E}_2), \mathbb{P}(\mathcal{E}_3)$$ in $K_0(\on{Var}_k)$.  The main result of this section is such a relationship.
\begin{thm}\label{motivicdecomposition}
Suppose $\mathcal{E}_1, \mathcal{E}_3$ have ranks $r_1, r_3$ respectively, so that $\mathcal{E}_2$ has rank $r_2:=r_1+r_3$.  Then $$[\mathbb{P}(\mathcal{E}_2)]=[\mathbb{P}(\mathcal{E}_1)]+\mathbb{L}^{r_1}[\mathbb{P}(\mathcal{E}_3)]=[\mathbb{P}(\mathcal{E}_3)]+\mathbb{L}^{r_3}[\mathbb{P}(\mathcal{E}_1)]\in K_0(\on{Var}_k).$$ 
\end{thm}
Before giving the proof, we need a lemma.
\begin{lem}\label{normallemma}
Let $S$ be a scheme and $$\mathcal{E}=\mathcal{E}_1\oplus \mathcal{E}_2$$ a split $\alpha$-twisted vector bundle on $S$.  Then $$\mathbb{P}(\mathcal{E})\setminus \mathbb{P}(\mathcal{E}_2)\simeq\on{Tot}(\mathcal{N}_{\mathbb{P}(\mathcal{E}_2)/\mathbb{P}(\mathcal{E})})$$ over $\mbb{P}(\mathcal{E}_1)$.
\end{lem}
\begin{proof}
The idea of this statement is that projection away from $\mbb{P}(\mathcal{E}_2)$ induces the desired isomorphism.  This is well known in the case that $\alpha\in H^2(S, \mbb{G}_m)$ is trivial; that is, in the case where the $\mathcal{E}_i$ are ordinary (untwisted) vector bundles.  We reduce to that case.  

Observe that the map $\mbb{P}(\mathcal{E}_1)\times \mbb{P}(\mathcal{E}_1)\to \mbb{P}(\mathcal{E}_1)$ admits a section (the diagonal map); thus by Corollary \ref{sbwithsection}, if $\pi_1: \mbb{P}(\mathcal{E}_1)\to S$ is the structure map, $$\pi_1^*\alpha=0\in H^2(\mbb{P}(\mcl{E}_1), \mbb{G}_m).$$
Thus in particular $\mbb{P}(\mcl{E})\times \mbb{P}(\mcl{E}_1), \mbb{P}(\mcl{E}_i)\times \mbb{P}(\mcl{E}_1)$ are trivial Severi-Brauer varieties over $\mbb{P}(\mcl{E}_1)$, so by the split case, we have that there is a natural isomorphism $$\on{Tot}(\mcl{N}_{\mbb{P}(\mcl{E}_1)\times \mbb{P}(\mcl{E}_1)/\mbb{P}(\mcl{E})\times \mbb{P}(\mcl{E}_1)})\simeq \mbb{P}(\mcl{E})\times\mbb{P}(\mcl{E}_1)\setminus \mbb{P}(\mcl{E}_2)\times \mbb{P}(\mcl{E}_1)$$
over $\mbb{P}(\mcl{E}_1)\times \mbb{P}(\mcl{E}_1)$.  Pulling back along the diagonal map $\Delta: \mbb{P}(\mcl{E}_1)\to \mbb{P}(\mcl{E}_1)\times \mbb{P}(\mcl{E}_1)$ gives the desired claim.
\end{proof}
\begin{proof}[Proof of Theorem \ref{motivicdecomposition}]
Without loss of generality, $S$ is integral and affine, and the short exact sequence $$0\to \mathcal{E}_1\to \mathcal{E}_2\to \mathcal{E}_3\to 0$$ splits (by Proposition \ref{semisimple}), so it suffices to prove the first equality, and we may view $\mathbb{P}(\mathcal{E}_1), \mathbb{P}(\mathcal{E}_3)$ as (linear) Severi-Brauer subvarieties of $\mathbb{P}(\mathcal{E}_2)$.

The morphism $\mathcal{E}_1\to \mathcal{E}_2$ induces a closed embedding $\mathbb{P}(\mathcal{E}_1)\hookrightarrow \mathbb{P}(\mathcal{E}_2)$, so $$[\mbb{P}(\mathcal{E}_2)]=[\mbb{P}(\mathcal{E}_1)]+[U],$$ where $U:= \mathbb{P}(\mathcal{E}_2)\setminus\mathbb{P}(\mathcal{E}_1)$.  We wish to identify $U$ with the total space of a vector bundle over $\mathbb{P}(\mathcal{E}_3)$.  But projection away from $\mathbb{P}(\mathcal{E}_1)$ identifies $U$ with the total space $\on{Tot}(\mathcal{N}_{\mathbb{P}(\mathcal{E}_3)/\mathbb{P}(\mathcal{E}_2)})$ of $\mathcal{N}_{\mathbb{P}(\mathcal{E}_3)/\mathbb{P}(\mathcal{E}_2)}$ by Lemma \ref{normallemma}.  $\on{Tot}(\mathcal{N}_{\mathbb{P}(\mathcal{E}_3)/\mathbb{P}(\mathcal{E}_2)})$ is a Zariski-locally trivial $\mathbb{A}^{r_1}$ fiber-bundle over $\mathbb{P}(\mathcal{E}_3)$, so $$[U]=[\on{Tot}(\mathcal{N}_{\mathbb{P}(\mathcal{E}_3)/\mathbb{P}(\mathcal{E}_2)})]=\mathbb{L}^{r_1}[\mathbb{P}(\mathcal{E}_3)]\in K_0(\on{Var}_k)$$ as desired.
\end{proof}
\begin{cor}\label{motdec}
Suppose $\mathcal{E}$ is an $\alpha$-twisted vector bundle with $\mcl{E}=\mathcal{F}^{\oplus n}$ for some $\alpha$-twisted vector bundle $\mathcal{F}$ of rank $r$.  Then $$[\mbb{P}(\mcl{E})]=[\mbb{P}(\mcl{F})](1+\mbb{L}^{r}+\cdots+\mbb{L}^{r(n-1)}).$$
\end{cor}
\begin{proof}
This is immediate from Theorem \ref{motivicdecomposition} and induction on $n$.
\end{proof}
\begin{prop}\label{independence}
Let $S$ be a $k$-variety and $P_1, P_2$ two Severi-Brauer varieties over $S$ of the same dimension and with the same Brauer class $\alpha$.  Then $$[P_1]=[P_2]\in K_0(\on{Var}_k).$$ 
\end{prop}
\begin{proof}
We may immediately replace $S$ with $S_{\on{red}}$.  Suppose $\mathcal{E}_1, \mathcal{E}_2$ are $\alpha$-twisted sheaves with $P_i=\mathbb{P}(\mathcal{E}_i)$.  Then by Corollary \ref{monomorphism} (replacing $S$ with an integral affine subscheme), there is an open set $U\subset S$ so that $\mathcal{E}_1|_U\simeq \mathcal{E}_2|_U$.  Thus $P_1|_U\simeq P_2|_U$ and so $[P_1|_U]=[P_2|_U]$.  Proceed by Noetherian induction.
\end{proof}
\begin{prop}[Theorem \ref{sbresult} refined]\label{refined9}
Let $S$ be a $k$-variety and $\alpha\in Br(S)$ a Brauer class.  Let $r=\on{gcd}(\on{rk}(\mcl{E}))$, where $\mcl{E}$ runs over all $\alpha$-twisted vector bundles.  Then there exists a class $P\in K_0(\on{Var}_k)$ so that for any Severi-Brauer $S$-scheme $\mbb{P}(\mcl{E})$ with Brauer class $\alpha$ and $\on{rk}(\mcl{E})=d$, $$[\mbb{P}(\mcl{E})]=P(1+\mbb{L}^r+\mbb{L}^{2r}+\cdots+ \mbb{L}^{d-r}).$$
\end{prop}
\begin{proof}
We first show that given $\mcl{E}$, there exists a $P$ as desired; then we show that the class of $P$ does not depend on $\mcl{E}$.  

By Corollary \ref{motdec}, it suffices to find a stratification $\{S_i\}$ of $S$ so that on each stratum $(S_i)_{\on{red}}$, $\mcl{E}|_{S_i}=\mcl{F}_i^{\oplus k}$ for some $\alpha$-twisted vector bundle $\mcl{F}_i$ of rank $r$ on $(S_i)_{\on{red}}$.  For then we may write $P=\sum_i [\mbb{P}(\mcl{F}_i)]$, and the result follows for $\mcl{E}$. 

Now let $S_1$ be any irreducible open affine; then at the generic point $\iota: \eta\hookrightarrow (S_1)_{\on{red}}$, $\mcl{E}|_\eta=D^{\oplus k}$ for the unique simple $\iota^*\alpha$-twisted vector bundle $D$.  But $\on{rk}(D)$ divides $r$, as the generic fiber of any $\alpha$-twisted vector bundle admits a similar decomposition, so after shrinking $S_1$ we may take $\mcl{F}_1=\mcl{D}^{\oplus k'}$ for some $k'$ and some $\mcl{D}$ extending $D$.  We now proceed by Noetherian induction.

To see that our choice of $P$ is independent of $\mcl{E}$, let $\mcl{E}'$ be another $\alpha$-twisted vector bundle, with associated stratification $\{S_j'\}$ and twisted vector bundles $\mcl{F}_j'$ on $(S_j')_{\on{red}}$, and $P'=\sum_j [\mbb{P}(\mcl{F}_j')]$.  Then on each irreducible component of $$U_{ij}=(S_i\cap S_j')_{\on{red}},$$ $\mbb{P}(\mcl{F}_i|_{U_{ij}}), \mbb{P}(\mcl{F}_j'|_{U_{ij}})$ satisfy the hypothesis of Proposition \ref{independence}.  Thus, $P=P'$ as desired.
\end{proof}
\begin{rem}
This result is an analogue of the main result of \cite{Kar95} with the features that (1) equality holds in the Grothendieck ring of varieties and (2) the result is proven in the relative setting.  The methods here may be used to obtain relative versions of the many of the results of \cite{zhykhovich}, \cite{Kar95}; for example, the main theorem \cite[Theorem 1.3.1]{Kar95}, which gives a decomposition of the motive of a Severi-Brauer variety over a field may be extended to Severi-Brauer schemes over arbitrary $k$-varieties.
\end{rem}
\section{The Abel-Jacobi Morphism}
Let $C$ be a smooth, projective, geometrically connected curve over a field $k$, with genus $g$.  We now consider the Abel-Jacobi morphism $$AJ^n: \on{Sym}^n(C) \to \on{Pic}^n(C)$$ where $n> 2g-2$.  If $C$ has a rational point, this is a Zariski $\mathbb{P}^{n-g}$-bundle; so in general, $AJ^n$ exhibits $\on{Sym}^n(C)$ as a Severi-Brauer variety over $\on{Pic}^n(C)$.  

Let $K/k$ be a finite separable extension over which $C$ obtains a rational point, so that there is a universal line bundle $\mathcal{L}_n$ over $C_K\times \on{Pic}^n(C)_K$, and let $p: C\times \on{Pic}^n(C)\to \on{Pic}^n(C)$ and $q: C\times \on{Pic}^n(C)\to C$ be the natural projections; we let $p_K, q_K$ be the maps obtained by extending scalars to $K$.  Then by \cite[Theorem 4]{schwarzenberger}, $$\on{Sym}^n(C)_K\simeq \mathbb{P}_{\on{Pic}^n(C)_K}({p_K}_*\mathcal{L}_n)$$ for $n> 2g-2$.  Viewing $\on{Sym}^n(C)$ as a descent of $\mathbb{P}_{\on{Pic}^n(C)_K}({p_K}_*\mathcal{L}_n)$ induces descent data on $\mathbb{P}_{\on{Pic}^n(C)_K}({p_K}_*\mathcal{L}_n)$, which we  may view as a $1$-cocycle valued in $\on{PGL}({p_K}_*\mathcal{L}_n)$.  Choosing an arbitrary lift of this $1$-cocycle to a $1$-cocycle valued in $\on{GL}({p_K}_*\mathcal{L}_n)$ (to do so, one may have to refine the cover $\on{Pic}^n(C)_K\to \on{Pic}^n(C)$), we may view ${p_K}_*\mathcal{L}_n$ as an $\alpha$-twisted sheaf $\mathcal{F}_n$ on $\on{Pic}^n(C)$ for some $\alpha\in H^2(\on{Pic}^n(C), \mathbb{G}_m)$, and $\on{Sym}^n(C)=\mathbb{P}_{\on{Pic}^n(C)}(\mathcal{F}_n)$.  
\begin{proof}[Proof of Theorem \ref{maintheorem}]
Let $g$ be the genus of $C$.

Let $D$ be a $k$-rational effective $0$-cycle on $C$ of degree $n$, i.e. a rational point of $\on{Sym}^n(C)$ for some $n$.  Let $f\in\Gamma(C, \mathcal{O}_C(D))$ be such that $$0\to \mathcal{O}_C(-D)\overset{\cdot f}{\longrightarrow} \mathcal{O}_C\to \mathcal{O}_D\to 0$$ is exact.  Let $$a^m_D: \on{Pic}^m(C)\overset{\sim}{\longrightarrow} \on{Pic}^{m+n}(C)$$ be the map induced by multiplication by $AJ^n(D)$.  Note that after changing base to $K$, there is an isomorphism ${a_D^m}^*\mathcal{L}_{m+n}\simeq \mathcal{L}_m\otimes q_K^*\mathcal{O}_C(D)$; as $f$ is defined over $k$, multiplication by $f$ induces a morphism $b_D^m: \mathcal{F}_m\to {a_D^m}^*\mathcal{F}_{m+n}$.  One may check that $b_D^m$ is a monomorphism by changing base to $K$.  For $m>2g-2$, the induced map $$\mathbb{P}(b_D^m): \on{Sym}^m(C)\simeq \mathbb{P}(\mathcal{F}_m)\to \mathbb{P}({a_D^m}^*\mathcal{F}_{m+n})\simeq \on{Sym}^{m+n}(C)$$ agrees with the morphism $\on{Sym}^m(C)\to \on{Sym}^{n+m}(C)$ sending a effective degree $m$ $0$-cycle $R$ to $R+D$.  Furthermore, the existence of the morphism $b_D^m$ implies that $\mathcal{F}_m, {a_D^m}^*\mathcal{F}_{n+m}$ are vector bundles twisted by the same class $[\alpha]\in H^2(\on{Pic}^m(C), \mathbb{G}_m)$.  

Let $R_m=\on{coker}(b_D^m)$; by extending scalars to $K$, we see that $R_m$ is an $\alpha$-twisted vector bundle of rank $n$.  Thus $$[\on{Sym}^{n+m}(C)]=[\mathbb{P}({a_D^m}^*\mathcal{F}_{n+m})]=[\mathbb{P}(\mathcal{F}_m)]+\mathbb{L}^{m-g+1}[\mathbb{P}(R_m)]=\mathbb{L}^{n}[\on{Sym}^m(C)]+ [\mathbb{P}(R_m)]$$ by Theorem \ref{motivicdecomposition}.  Observe that $R_m$ and ${a^m_D}^*R_{m+n}$ are $\alpha$-twisted vector bundles of the same rank; thus by Proposition \ref{independence}, $$[\mathbb{P}(R_m)]=[\mathbb{P}({a_D^m}^*R_{m+n})]=[\mbb{P}(R_{m+n})].$$  Let $[P_m]\in K_0(\on{Var}_k)$ be this class.  Then by induction, we have that 
$$[\on{Sym}^{m'+n}(C)]=[P_m]+\mathbb{L}^n[\on{Sym}^{m'}(C)]$$
for all $$m'\equiv m \bmod n,~m'> 2g-2.$$
Thus there exists a polynomial $p(t)\in K_0(\on{Var}_k)[t]$ such that 
$$Z_C(t)=p(t)+\mathbb{L}^nt^nZ_C(t)+\sum_{m=2g-1}^{2g+n-2}\frac{[P_m]t^m}{1-t^n}.$$
In particular, $$(1-\mathbb{L}^nt^n)(1-t^n)Z_C(t)$$ is a polynomial.  As $Z_C(t)$ has constant term $1$, so does $(1-\mathbb{L}^nt^n)(1-t^n)Z_C(t)$.
\end{proof}
\begin{cor}\label{mainthmcor}
Let $C$ be a curve over $k$ such that each irreducible component of  $\widetilde{C_{\on{red}}}$ (the normalization of the underlying reduced curve $C_{\on{red}}$) is geometrically irreducible.  Then there exists a polynomial $p(t)\in 1+tK_0(\on{Var}_k)[t]$ so that  $p(t)Z_C(t)$ is a polynomial with constant term $1$.
\end{cor}
\begin{proof}
We reduce to the case $C$ is smooth and projective. Indeed, we may assume $C$ is reduced as $[C]=[C_{\on{red}}]$; let $\tilde C$ be the smooth projective model of $C$.  Then $[C]=[\tilde C]+[X]-[Y]$, where $X, Y$ are zero-dimensional schemes.  In particular $$Z_C(t)Z_Y(t)=Z_{\tilde C}(t)Z_X(t)$$ by Remark \ref{mult}.  We leave it to the reader to show that there exist polynomials $p_X(t), p_Y(t)\in 1+tK_0(\on{Var}_k)[t]$ so that $$p_X(t)Z_X(t), p_Y(t)Z_Y(t)$$ are polynomials with constant term one; thus to prove the theorem for $C$ it suffices to prove it for $\tilde C$.  But $\tilde C$ is a disjoint union of components $C_i$ satisfying the conditions of Theorem \ref{maintheorem}, and $$Z_C(t)=\prod_i Z_{C_i}(t),$$ so we are done.
\end{proof}
\begin{rem}
It is natural to guess that the motivic zeta function of \emph{any} curve is rational; that is, one may drop the condition of geometric connectedness in Theorem \ref{maintheorem}, and the rather artificial hypothesis of Corollary \ref{mainthmcor}.
\end{rem}

\end{document}